\documentclass[12pt, reqno]{amsart}

\usepackage{tikz}

\usepackage{etex}
\usepackage{textcomp} 
\usepackage{graphicx} 
\usepackage{flafter}
\usepackage{verbatim}
\usepackage[all]{xy}
\usepackage{epsfig,enumerate}
\usepackage{algorithmic}
\usepackage{amsmath}
\usepackage{amsthm} %potrebbe dare problemi
\usepackage{amssymb}
\usepackage{mathrsfs}
\usepackage{amsfonts}
\usepackage{amscd}
\usepackage{euscript}
\usepackage{amssymb}
\usepackage{mathrsfs}
\usepackage{marvosym}
\usepackage[nice]{nicefrac}
\usepackage{bm}
\usepackage{parallel}
\usepackage[pdftex,bookmarks,breaklinks]{hyperref} %colorlinks
\numberwithin{equation}{section}
\usepackage{memhfixc}
\usepackage{pdfsync}
\usepackage{microtype}
\usepackage{booktabs}
\usepackage{setspace}
\usepackage{fancyhdr}
\usepackage{indentfirst}
\usepackage{footnote}
\usepackage{booktabs}
\usepackage{listings}
\usepackage{enumitem}
\usepackage{xcolor}
\usepackage{scrextend}
\usepackage{chngcntr}
\usepackage[utf8]{inputenc}
\usepackage{lmodern}

\DeclareFontEncoding{OT2}{}{} % to enable usage of cyrillic fonts

\newtheorem{lemma}{Lemma}[section]

\newtheorem*{theorem*}{Theorem}

\newtheorem{cor}[lemma]{Corollary}

\newtheorem{claim*}{Claim}
\newtheorem{thm}[lemma]{Theorem}

\theoremstyle{definition}
\newtheorem{remark}[lemma]{Remark}

\newtheorem{question}[lemma]{Question}
\newtheorem{defn}[lemma]{Definition}

\newcommand{\PP}{{\mathbb P}}

\newcommand{\Z}{{\mathbb Z}}

\newcommand{\frakD}{{\mathfrak D}}

\newcommand{\frakP}{{\mathfrak P}}

\numberwithin{equation}{section}
\numberwithin{table}{section}

\newcommand{\defi}[1]{\textsf{#1}} % for defined terms

\usepackage[margin=2.3cm]{geometry}

\title{Degrees of closed points on hypersurfaces}

\author{F. Balestrieri}
\address{The American University of Paris,  5 Boulevard de La Tour-Maubourg, 75007 Paris, France}
\email{fbalestrieri@aup.edu}
\date{\today}

\begin{document}
\maketitle

\begin{abstract} Let $k$ be any field. Let $X \subset \PP_k^N$ be a degree $d \geq 2$ hypersurface. Under some conditions, we prove that if $X(K) \neq \emptyset$ for some extension $K/k$ with $n:=[K:k] \geq 2$ and $\gcd(n,d)=1$, then $X(L) \neq \emptyset$ for some extension $L/k$ with $\gcd([L:k], d)=1$, $n \nmid [L:k]$, and $[L:k] \leq nd-n-d$. Moreover,  if a $K$-solution is known explicitly, then we can compute $L/k$ explicitly as well. As an application, we improve upon a result by Coray on smooth cubic surfaces  $X \subset \mathbb{P}^3_k$ by showing that if  $X(K) \neq \emptyset$ for some extension $K/k$ with $\gcd([K:k], 3)=1$, then $X(L) \neq \emptyset$ for some $L/k$ with $[L:k] \in \{1, 10\}$.
\end{abstract}

\section{Introduction}

Springer's theorem for quadratic forms famously states that, if a quadratic  form $\varphi$ on a finite-dimensional vector space over a field $k$ is isotropic over some extension $K
/k$ of odd degree, then it is already isotropic over $k$  (see \cite{Springer} for the case when the characteristic is not 2 and \cite[Corollary 18.5]{EKM} for any characteristic). Equivalently, in more geometric terms,  if $X \subset \PP^N_k$ is a degree $2$ hypersurface, then $X(K) \neq \emptyset$ for some extension $K/k$ of odd degree implies that $X(k) \neq \emptyset$.
A natural question to ask is whether Springer's theorem generalises to higher degree forms. 

\begin{question} \label{q1}Given a degree $d \geq 3$ hypersurface $X \subset \PP^N_k$ over a field $k$, is it true that if $X(K) \neq \emptyset$ for some extension $K/k$ with $\gcd([K:k],d)=1$, then $X(k) \neq \emptyset$?
\end{question}

When $d \geq 4$, the answer to Question \ref{q1} is likely to be \emph{no} in general (see e.g. \cite[Example 2.8]{Coray} for a counterexample when $d=4$), while, when $d=3$, Cassels and Swinnerton-Dyer have conjectured that the answer to Question \ref{q1} should instead be \emph{yes}. Some progress towards the conjecture by Cassels and Swinnerton-Dyer has been obtained by Coray (see \cite{Coray}), who proved, for any smooth cubic surface $X \subset \mathbb{P}_k^3$ over a perfect field $k$, that if $X(K) \neq \emptyset$ for some extension $K/k$ with $\gcd([K:k],3)=1$, then $X(L) \neq \emptyset$ for some extension $L/k$ with $[L:k] \in \{1,4,10\}$. In recent work, Ma has been able to remove the condition on the field being perfect, proving Coray's result for any field (see \cite{Ma}). Moreover, when $k$ is a global field, Rivera and Viray have shown  that, if the Brauer-Manin obstruction is the only one for the Hasse principle for rational  points on smooth cubic surfaces in $\PP^3$ over $k$ (and, by a conjecture by Colliot-Thélène and Sansuc \cite{CTS79}, this should always be the case), then the conjecture by Cassels and Swinnerton-Dyer holds for such surfaces (see \cite{RV}).

In this paper, we are concerned with the following much weaker version of Question \ref{q1}.

\begin{question} \label{q2}Let  $X \subset \PP^N_k$ be a degree $d \geq 3$ hypersurface over a field $k$. If $X(K) \neq \emptyset$ for some finite extension $K/k$ with $\gcd([K:k], d)=1$, can we find some (somewhat explicit) finite extension $L/k$ with $\gcd([L:k], d)=1$, $ [K:k] \nmid [L:k]$, and $X(L) \neq \emptyset$?
\end{question}

Our main theorem answers Question \ref{q2} positively under some assumptions on $d$ and $[K:k]$.

\begin{theorem*}[Theorem \ref{main}] Let $k$ be any field. Let $X \subset \PP^N_k $ be a degree $d \geq 2$ hypersurface.  Suppose that $X(K) \neq \emptyset$  for some simple extension $K/k$ with $n:=[K:k] \geq 2$ and $\gcd(n,d)=1$.
If $\frakP_{\textrm{bad}}(nd-n-d) = \emptyset$ (see Definition \ref{bad}), then $X(L) \neq \emptyset$  for some extension $L/k$ with $\gcd([L:k], d) = 1$, $n \nmid [L:k]$, and $[L:k] \leq nd -n-d$, where a set $\frakD$ of possible values for $[L:k]$ is explicitly computable. Moreover, if a point in $X(K) \neq \emptyset$   is known explicitly, then $L/k$ can be explicitly computed  as well.  \end{theorem*}

As a corollary of the above theorem, we can improve upon Coray's and Ma's results.

\begin{theorem*}[Theorem \ref{34}]
 Let $k$ be a field and let $X \subset \PP^N_k$ be a cubic hypersurface over  $k$.  If $X(K) \neq \emptyset$ for some simple field extension $K/k$ with  $[K:k] = 4$, then $X(L) \neq \emptyset$ for some extension $L/k$ with $[L:k] \in \{1,5\}$.
\end{theorem*}

\begin{cor} Let $k$ be a field and let $X \subset \mathbb{P}^3_k$ be a smooth  cubic surface over $k$. If $X(K) \neq \emptyset$ for some  extension $K/k$ with $\gcd([K:k],3)=1$, then $X(L) \neq \emptyset$ for some $L/k$ with $[L:k] \in \{1,10\}$.
\end{cor}
\begin{proof} By Coray's and Ma's results, we know, under the hypotheses of the corollary, that there exists some $L/k$ with $[L:k] \in \{1,4,10\}$ and $X(L) \neq \emptyset$. If $[L:k]=4$, then either $L/k$ is simple, in which case, by Theorem \ref{34},  there is some other $L'/k$ with $[L':k] \in \{1,5\}$ and $X(L') \neq \emptyset$, or $L/k$ is not simple. If $L/k$ is not simple, then, since it is finite, it must be a tower of simple extensions $L/k(\alpha)/k$ with $[L: k(\alpha)] = [k(\alpha):k]=2$. Then $X_{k(\alpha)}$ is a smooth cubic surface as well, and $X_{k(\alpha)} (L) \neq \emptyset$, where $[L:k(\alpha)]=2$; this implies that $ X(k(\alpha)) \neq \emptyset$. Repeating the same argument with $k(\alpha)$ and $k$, we get that $X(k) \neq \emptyset$ and we can thus let $L' =k$. 
In any case, we have found some $L'/k$ with $[L':k] \in \{1,5\}$ and $X(L') \neq \emptyset$. If $L'=k$ we are done, and if $[L':k]=5$, then any quadratic extension $L''/L'$ (thus with $[L'':k]=10$) satisfies $X(L'') \neq \emptyset$.
\end{proof}

Finally, when $d$ and $n$ are both primes, we obtain the following refinement of Theorem \ref{main}.

\begin{theorem*}[Corollary \ref{mainprime}] Let $k$ be any field. Let $X \subset \PP^N_k $ be a degree $d \geq 3$ hypersurface, with $d$ prime.  Suppose that $X(K) \neq \emptyset$  for some  extension $K/k$ with $n:=[K:k] \geq 2$ prime and $\gcd(n,d)=1$. Then $X(L) \neq \emptyset$  for some extension $L/k$ with $\gcd([L:k], nd) = 1$, and $[L:k] \leq nd -n-d$, where a set $\frakD$ of possible values for $[L:k]$ is explicitly computable. Moreover, if a point in $X(K) \neq \emptyset$   is known explicitly, then $L/k$ can be explicitly computed  as well.  \end{theorem*}

\begin{remark} When taking into account Springer's theorem for the case when $d=2$ as well, the statement of Theorem \ref{mainprime} then becomes completely symmetric in $n$ and $d$.
\end{remark}

\section*{Acknowledgments} The author is very thankful to both the anonymous referees for their insightful comments and suggestions, which greatly improved both the scope and the exposition of the arguments in the paper, and for pointing out some inaccuracies in a previous version of the paper. The key Remark \ref{alreadyiso} was observed by one of the anonymous referees, while a simplification of  the proof of Lemma \ref{admiss} (with a consequent slight strengthening of its statement) was suggested by the other anymous referee.

%%%%%%%%%%%%%%%%%%55

\section{Preliminaries on degree $d$ forms}

Hypersurfaces  $X \subset \PP^N_k$ of degree $d$ over a field $k$ are equivalent to degree $d$ (homogeneous) forms in $N+1$ variables over $k$. Since we are going to prove our main theorems in the language of forms, we start  by recalling some basic definitions.

\begin{defn} Let $\varphi $ be a form of degree $d$ on a finite-dimensional vector space $V $ over a field $k$. For any field extension $K/k$, we let the \defi{extension $\varphi_K$ of $\varphi$} be the degree $d$ form on $V_K := V \otimes_k K$ defined by 
\[ \varphi_K(v \otimes w) := w^d \varphi(v),\]
for all $v  \in V$ and $w \in K$. Note that if $K=k$, then we omit the subscript and write $\varphi_k = \varphi$. 
\end{defn}
\begin{defn} Let $\varphi $ be a form of degree $d$ on a finite-dimensional vector space $V$ over a field $k$.  We say that $\varphi$ is \defi{isotropic over $K$}  if $\varphi_K$ is isotropic, i.e,  if there exists some non-zero $w \in V_K$ with $\varphi_K(w) = 0$; otherwise, we say that $\varphi$ is \defi{anisotropic over $K$}.
\end{defn}
\begin{remark} If $X\subset \PP^N_k$ is a degree $d$ hypersurface over a field $k$ corresponding to the degree $d$ form $\varphi$ on $k^{N+1}$, then, for any extension $K/k$, we have that $X(K) \neq \emptyset$ if and only if $\varphi_K$ is isotropic.
\end{remark}

If $(i_0, ..., i_N) \in \Z_{\geq 0}^{N+1}$ and $x := (x_0, ..., x_N)$, we denote by $\underline{x}^{(i_0, ..., i_N)}$ the monomial  in which  $x_j$ appears with exponent $i_j$ if  $i_j >0$ and does not appear at all if $i_j=0$. In what follows, we will sometimes need to focus on a special type of forms, which we call \emph{diagonal-full}.

\begin{defn}\label{diagfull} Let $\varphi$ be a form of degree $d$ on a finite-dimensional vector space $V \cong k^{N+1}$ over a field $k$, say
\[ \varphi(x_0, ..., x_{N}) = \sum_{\substack{(i_0, ..., i_N) \in \Z_{\geq 0}^{N+1}:\\ i_0 + ... + i_N = d}} a_{(i_0, ..., i_N)} \underline{x}^{(i_0, ..., i_N)},\]
with $ a_{(i_0, ..., i_N)} \in k$. We say that $\varphi$ is \defi{diagonal-full} if $a_{(d,0, ..., 0)}, a_{(0,d,0, ..., 0)}, ..., a_{(0, ...,0 , d)} \neq 0$.\\
In more geometric terms, a degree $d$ hypersurface $X \subset \PP_k^{N}$ is \defi{diagonal-full} if $X$ is given by an equation of the form
\[  \sum_{\substack{(i_0, ..., i_N) \in \Z_{\geq 0}^{N+1}:\\  i_0 + ... + i_N = d}}  a_{(i_0, ..., i_N)} \underline{x}^{(i_0, ..., i_N)} = 0\]
with  $ a_{(i_0, ..., i_N)} \in k$ and $a_{(d,0, ..., 0)}, a_{(0,d,0, ..., 0)}, ..., a_{(0, ...,0 , d)} \neq 0$.

\end{defn}

\begin{remark} \label{alreadyiso}If a degree $d$ form $\varphi = \varphi(x_0, ..., x_N)$ on a finite-dimensional vector space $V \cong k^{N+1}$ over a field $k$ is \emph{not} diagonal-full, then it is already isotropic over $k$:  if, say, the monomial $x_i^d$ is missing, then
\[v := (0,...,0,\underbrace{1}_{\textrm{$i$-th coordinate}},0,...,0) \in V\] 
 is a non-trivial zero of $\varphi$. This implies that, in much of what follows, we can restrict our attention to diagonal-full forms.
\end{remark}

\begin{defn} Let $\varphi$ be a form on a finite-dimensional vector space $V \cong k^{N+1}$ over a field $k$. If $K/k$ is a field extension, we let 
\[D(\varphi_K) := \{  \varphi_K(w) : w \in V_K \cong K^{N+1},   \varphi_K(w)  \neq 0 \} \subset K^{\times}.\]
We denote by $\langle D(\varphi_K) \rangle$ the multiplicative set in $K^\times$ whose elements are finite products of elements in $D(\varphi_K)$.
\end{defn}
The following is a straightforward modification of \cite[Theorem 18.3, proof of  $(2) \Rightarrow (3)$]{EKM}.

\begin{lemma} \label{lemma1} Let $\varphi$ be a form of degree $d$ on a finite-dimensional vector space $V$ over $k$ and let $f \in k[t]$ be a non-constant polynomial.  If there exists some $a \in k^\times$ such that $ a f \in \langle D(\varphi_{k(t)}) \rangle$, then $\varphi_{k(p)}$ is isotropic for each irreducible polynomial $p$ occurring to a power coprime to $d$ in the factorisation of $f$, where $k(p) := k[t]/(p(t))$.
\end{lemma}

\begin{proof}
Since $af \in  \langle D(\varphi_{k(t)}) \rangle$, there exist some $0 \neq h \in k[t]$ and $v_1, ..., v_m \in V[t]$ such that
\[ a f h^d = \prod_{i=1}^m \varphi(v_i). \]
If it exists, let $p \in k[t]$ be a non-constant monic irreducible factor of $f$  appearing  with exponent $\lambda$ coprime to $d$ in the factorisation of $f$ into irreducible polynomials, i.e. say $f = p^\lambda f'$ with $p$ monic irreducible, $\deg(p) \geq 1$,  $p \nmid f'$, and $\gcd(\lambda, d) = 1$. Write $v_i = p^{k_i} v_i'$, where $k_i \geq 0$ and $p \nmid v'_i$, for each $i = 1, ...,m$. Then
\[a p^\lambda f' h^d =  a f h^d = \prod_{i=1}^m \varphi(v_i) = \prod_{i=1}^m p^{d k_i} \varphi(v_i') = p^{d \sum_{i=1}^m k_i} \prod_{i=1}^m  \varphi(v_i') . \]
Since 
\[ \lambda + d \nu_p(h) =  \nu_p(a p^\lambda f' h^d) = \nu_p\left(\prod_{i=1}^m p^{d k_i} \varphi(v_i') \right)= d \sum_{i=1}^m k_i + \sum_{i=1}^m  \nu_p( \varphi(v_i') ), \]
where $\nu_p(-)$ denotes the valuation at $p$, and since $\gcd(\lambda, d) = 1$, it follows that $\nu_p(\varphi(v_j')) \geq 1$ for some $j \in \{1, ..., m\}$. This means that  $\varphi(v_j') \equiv 0 \bmod p$. Since by construction $p \nmid v_j'$, we also have that $v_j' \not\equiv 0 \bmod p$. Hence, $\varphi_{k(p)}$ is isotropic, as required.
\end{proof}

\begin{lemma} \label{degr} Let $d$ be a positive integer. Let $k$ be a field and let $\varphi$ be a diagonal-full form of degree $d$ on a finite-dimensional vector space $V \cong k^{N+1}$ over $k$. Suppose that $\varphi$ is anisotropic. Let $0 \neq s := (s_0, ..., s_N)  \in V[t]$. Then $\deg(\varphi(s)) = d \deg(s)$, where $deg(s):=\max_{i=0, ..., N}(\deg(s_i))$.
\end{lemma}
\begin{proof}  Let 
\[ I_{\deg(s)} := \{ i \in \{0, ..., N\}: \deg(s_i) = \deg(s)\}.\]

Since $\varphi$ is diagonal-full, we can write it as
\[ \varphi(x_0, ..., x_N) = \sum_{\substack{(i_0, ..., i_N) \in \Z_{\geq 0}^{N+1}: \\  i_0 + ... + i_N = d}} a_{(i_0, ..., i_N)} \underline{x}^{(i_0, ..., i_N)},\]
with $ a_{(i_0, ..., i_N)} \in k$ and $a_{(d,0, ..., 0)}, a_{(0,d,0, ..., 0)}, ..., a_{(0, ...,0 , d)} \neq 0$.
If $\deg(\varphi(s)) \neq d \deg(s)$, then some cancellation must have occured among the leading coefficients (not all 0, since $\varphi$ is diagonal-full)  of those polynomials $ a_{(i_0, ..., i_N)} \underline{s(t)}^{(i_0, ..., i_N)}$ of degree $d \deg(s)$. (We note that, since $d \deg(s)$ is the maximal degree that can possibly be attained, the polynomial $\underline{s(t)}^{(i_0, ..., i_N)}$  has degree $d \deg(s)$ if and only if $i_j=0$ for all $j \notin I_{\deg(s)}$.) In particular, if we let $0 \neq \tilde{s} \in k^{N+1} \cong V$ be defined by
\[ \tilde{s}_i = 
\begin{cases}
s^\ast_i & \textrm{ if } i \in I_{\deg(s)},\\
0  & \textrm{ if } i \notin I_{\deg(s)},\\
\end{cases}
\]
where $s^\ast_i \in k$ denotes the leading coefficient of $s_i(t)$, then $\tilde{s}$ must satisfy
\[\varphi(\tilde{s}) =  \sum_{\substack{(i_0, ..., i_N) \in \Z_{\geq 0}^{N+1}:\\  i_0 + ... + i_N = d}} a_{(i_0, ..., i_N)} \underline{\tilde{s}}^{(i_0, ..., i_N)}= 0,\]
 which would imply that $\varphi$ is isotropic, a contradiction.
Hence, $\deg(\varphi(s)) = d \deg(s)$, as required.
\end{proof}

\begin{lemma} \label{stayaniso} Let $k$ be any field and let $\varphi$ be a degree $d$ form on a finite-dimensional vector space $V \cong k^{N+1}$ over $k$. If $\varphi$ is anisotropic over $k$, then $\varphi_{k(t)}$ is anisotropic over the function field $k(t)$.
\end{lemma}
\begin{proof} Assume, for a contradiction, that $\varphi_{k(t)}$ is isotropic.  Let $w = (w_0, ..., w_N) \in k(t)^{N+1}$ be a non-zero vector with $\varphi_{k(t)  }(w)=0$.
 Let $h \in k(t)$ be such that $hw :=(hw_0, ..., h w_N)$ satisfies $h w_i \in k[t]$ for all $i = 0, ..., N$ and $\gcd(hw_0, ..., hw_N) =1$. Then 
$\varphi_{k(t)  }(hw) = h^d \varphi_{k(t)  }(w) = 0$. In particular, if we specialise $t = a$ for some $a \in k$, we also have $\varphi((h(a)w_0(a), ..., h(a)w_N(a))) = 0$. Since $\varphi$ is anisotropic over $k$, this implies that $h(a)w_i(a) = 0$ for all $i = 0, ..., N$. But this, in turn, implies that $(t-a) \mid h(t)w_i(t)$ for all $i = 0, ..., N$, which is a contradiction to $\gcd(hw_0, ..., hw_N) =1$. Hence, $\varphi_{k(t)}$ must be anisotropic, as required.
\end{proof}

\section{Proof of the main theorems}
In this section, using fairly simple arguments, we prove (in the language of forms) the main theorems of the paper. Despite these arguments being generalisations to higher degree forms of some of the ideas in the proof of Springer's theorem, to the best of the author's knowledge they have not appeared  anywhere in the literature before and are thus to be considered novel.

\subsection{The case for general $d$ and $n$.}  Before we can prove our first main theorem, we need some preliminary definitions and results on partitions.

\begin{defn} Let $d, n \geq 2$ be positive integers. Let $n^\ast \in \{1, ..., d-1\}$ be the unique integer such that  $n^\ast \equiv - n \bmod d.$
We define the set
\[ S_{d, n} := \left\{ n^\ast + j d : j \in \Z_{\geq 0} \textrm { and }  n^\ast + j d < n(d-1) \right\}.\]
\end{defn}

\begin{defn} Let $u \in \Z_{>0}$. By a \defi{partition of $u$} we mean a non-empty multiset $[a_1, ..., a_r]$ with $r \in \Z_{>0}$, $a_i \in \Z_{>0}$, and $u = \sum_{i=1}^r a_i$.  By a \defi{subpartition of a partition $[a_1, ..., a_r]$ of $u$} we mean a non-empty multisubset of  $[a_1, ..., a_r]$.
We let $\frakP(u)$ denote the set of all partitions of $u$.
\end{defn}
We will need to characterise the maximal element in $S_{d,n}$ and its partitions.

\begin{lemma} \label{eqB} For any positive integers $n,d \geq 2$ with $\gcd(d,n)=1$ we have 
\[ \max S_{d,n} = nd-n-d.\]
\end{lemma}
\begin{proof} We assume first that $n < d$.  If $n^\ast \in \{1, 2, ..., d-1\}$ is such that $n^\ast \equiv -n \bmod d$, then, since $n<d$, we have $n^\ast = d-n$.
Hence, 
\[ \begin{array}{ll}
S_{d,n} &= \{ d-n + jd : j \in \Z_{\geq 0}  \textrm{ and } d-n + j d <n (d-1) \} \\
&= \{ d-n + jd : j \in \{0,1, ..., n-2\} \}\\
\end{array}\]
and so $\max S_{d,n} = d - n + (n-2)d = dn -n - d$.

Assume now that $d < n$. If $n^\ast \in \{1, 2, ..., d-1\}$ is such that $n^\ast \equiv -n \bmod d$, then, since $d<n$, we can write $n^\ast = \alpha d - n$ where $\alpha$ is the unique positive integer strictly between $\frac{n}{d}$ and $\frac{d+n}{d}$.
Hence, 
\[ \begin{array}{ll}
S_{d,n} &= \{  \alpha d - n + jd : j \in \Z_{\geq 0}  \textrm{ and } \alpha d - n + j d < n (d-1) \} \\
&= \{  \alpha d - n + jd : j \in \{0,1, ..., n-\alpha -1 \} \}\\
\end{array}\]
and so $\max S_{d,n} = \alpha d - n + ( n-\alpha -1)d  = dn -n - d$.

So, in any case,   $\max S_{d,n}  = dn -n - d$, as required. 
\end{proof}

\begin{lemma}\label{subpar} Let $n,d \geq 2$ be integers with $\gcd(d, n)=1$. Let $u \in S_{d,n}$. If  $u \neq nd - n -d$ and  $[a_1, ..., a_r]$ is a partition of $u$, then $[a_1, ..., a_r]$ is a subpartition of the partition $[a_1, ..., a_r, \mu d]$ of $dn -n - d$, where $\mu := \frac{nd-n-d - u}{d} \in \Z_{>0}$.
\end{lemma}
\begin{proof} Write $u = n^\ast + jd$ for some $j \in \Z_{\geq 0}$ and $dn -n -d = n^\ast +  j'd$ for some $j' \in \Z_{\geq 0}$ with $j' > j$. By writing
$ nd-n-d = u + (j'-j)d$ and by simple computations, the result follows.
\end{proof}

\begin{defn} \label{bad}Let $n, d \geq 2$ be integers with $\gcd(d,n)=1$. Let $u \in \Z_{0}$. We say that a partition $[a_1, ..., a_r] \in \frakP(u)$  is \defi{bad} if any of the following two cases holds:
\begin{itemize}
\item[(i)]  for all $i \in \{1, ..., r\}$, we have that $\gcd(a_i, d)>1$;
\item[(ii)] for all $i \in \{1, ..., r\}$ with $\gcd(a_i,d)=1$, we have that $n \mid a_i$.
\end{itemize}
We call the partition \defi{good} otherwise. We let 
\[\frakP_{\textrm{bad}}(u) \subset \frakP(u)\]
 be the (potentially empty) set of bad partitions of $nd-d-n$.
\end{defn}

We can now prove the main theorem of this paper, which answers Question \ref{q2} in certain cases.

\begin{thm} \label{main} Let $k$ be any field. Let $\varphi $ be a degree $d \geq 2$ form on a finite-dimensional vector space $V$ over a field $k$.  Suppose that $\varphi_K$ is isotropic for some simple extension $K/k$ with $n:=[K:k] \geq 2$ and $\gcd(n,d)=1$.
If $\frakP_{\textrm{bad}}(nd-n-d) = \emptyset$, then $\varphi_L$ is also isotropic for some extension $L/k$ with $\gcd([L:k], d) = 1$, $n \nmid [L:k]$, and $[L:k] \leq nd -n-d$, where a set $\frakD$ of possible values for $[L:k]$ is explicitly computable. Moreover, if a non-trivial solution for $\varphi_K$  is known explicitly, then $L/k$ can be explicitly computed  as well.  \end{thm}

\begin{proof} 
If $\varphi$ is isotropic over $k$, we can take $L=k$. So, from now on, we assume that $\varphi$ is anistropic over $k$. Moreover, by Remark \ref{alreadyiso}, we can also assume that $\varphi$ is diagonal-full, since otherwise it is already isotropic over $k$.

Since $K/k$ is a simple extension, we let $K = k(\alpha)$ and let $f \in k[t]$ be the minimal (irreducible) polynomial of $\alpha$ over $k$. Since, by assumption, $\varphi_{k(f)}$ is isotropic, it follows that there exists some $v \in V[t]$ such that $\varphi(v) \equiv 0 \bmod f$ but $v \not\equiv 0 \bmod f$.
By the division algorithm, there exist some $w,s \in V[t]$ such that
\[ v = fw + s\]
and  $\deg(s) < \deg(f) =n$.
Since
\[ \varphi(v) = \varphi(fw+s) = f^d \varphi(w) + f (\textrm{other stuff}) + \varphi(s)\]
and since $f \mid \varphi(v)$, it follows that $f \mid \varphi(s)$. If $s=0$, then $f \mid v$, which contradicts $f \nmid v$.  Hence, $s \neq 0$. Let $\varphi(s) = fg$ for some $g \in k[t]$. Since $s \neq 0$ and since, by assumption, $\varphi$ is anisotropic, we have by Lemma \ref{stayaniso} that $\varphi(s) \neq 0$. It follows that $g \neq 0$. Hence, we have that $fg \in \langle D(\varphi_{k(t)}) \rangle$.
Since $\varphi(s) = fg$ and $\deg(s) < \deg(f)$, it follows that
\[ \deg(f) + \deg(g) = \deg(\varphi(s)) < d \deg(f) = dn,\]
that is, \( \deg(g) < n(d-1).\) Notice also that $\deg(g) \geq 1$, since otherwise we would get, by Lemma \ref{degr}, that $d \deg(s) = \deg(\varphi(s)) = \deg(f)=n$, which is a contradiction to the fact that $\gcd(d,n) =1$.

In the remainder of the proof, we aim to show that there exists an irreducible factor $p$ of $fg$ of exponent $\lambda$ coprime to $d$ and with $\gcd(\deg(p), d)=1$, $n \nmid \deg(p)$, and $\deg(p) > 1$, with the goal of then applying Lemma \ref{lemma1} to it. Let the factorisation of $g$ into irreducible factors be 
\[g = g^\ast \prod_{i=1}^r p_i^{\lambda_i}\]
where $g^\ast \in k^\times$ and, for each $i = 1, ..., r$,  the distinct polynomials $p_i \in k[t]$ are monic and irreducible, with  $\deg(p_i) =: u_i$ and  $\lambda_i \geq 1$.
Then
\[ \deg(g) = \sum_{i=1}^r u_i \lambda_i < n(d-1).\]

We claim that $\deg(g) \in S_{d,n}$. Indeed, by Lemma \ref{degr}, we have that $\deg(\varphi(s)) = d \deg(s)$. Since $\varphi(s) = fg$, it follows that $\deg(g) = -n + d \deg(s)$. Since, moreover, $1 \leq \deg(g) < n(d-1)$, it follows that $\deg(g) \in S_{d,n}$, as claimed.

By Lemma \ref{eqB}, we then know that $\deg(g) \leq \max S_{d,n} = nd-n-d$. Moreover, by Lemma \ref{subpar}, if $\deg(g) \neq nd-n-d$, then we have that $[\lambda_1 u_1, ..., \lambda_r u_r, jd]$ is a partition of $nd-n-d$, for some $j \in \Z_{\geq 1}$. Since, by hypothesis, $\frakP_{\textrm{bad}}(nd-n-d) = \emptyset$, it follows that there is some $i \in \{1, ..., r\}$ with $\gcd(\lambda_i u_i, d) =1$ and $n \nmid \lambda_i u_i$. In particular, for such an $i$, we have $\gcd(\lambda_i, d)=1$, $\gcd(u_i, d)=1$, and $n \nmid u_i$.  
This corresponds to an irreducible factor $p_i $ of degree $u_i$ of $g$  with exponent $\lambda_i$ coprime to $d$. We notice that $p_i \nmid f$, since both $f$ and $p_i$ are irreducible and $\deg(p_i) = u_i \neq n = \deg(f)$, since $n \nmid u_i$. Hence, $p_i$ is a monic irreducible factor of $fg$ of exponent $\lambda_i$ coprime to $d$. By Lemma \ref{lemma1}, this implies that $\varphi_{k(p_i)}$ is isotropic. By letting $L:= k(p_i) = k[t]/(p_i(t))$, we see that $[L:k] = u_i$ satisfies $\gcd([L:k], d)=1$ and $n \nmid [L:k]$. Since $\deg(g) \leq nd-n-d$, it also follows immediately, by construction, that $[L:k] \leq nd-n-d$, as required.\\

In order to explicitly compute a set $\frakD$ of possible values for $[L:k]$, we can do as follows. Let $[a_1, ..., a_r]$ be a partition of $nd-n-d$.  For each $a_i$ with $\gcd(a_i, d)=1$, potentially we could have $a_i = u_i \lambda_i$, and thus $u_i$ could be any positive divisor of $a_i$.  Hence, for each $a_i$ with $\gcd(a_i,d)=1$, any positive divisor $\delta$ of $a_i$ with $n \nmid \delta$ is a possible degree $[L:k]$ for the extension constructed above with the required properties, and we append any such $\delta$ to $\frakD$. By considering all the possible partitions of $nd-n-d$ and appeding divisors $\delta$ to $\frakD$ as above, we thus obtain a set $\frakD$ containing, by construction, the value of $[L:k]$.\\

Finally, if we have an explicit non-trivial solution over $K$, then, in the above proof, we also have an explicit $v \in V[t]$, which implies that $w,s$ are also explicit, and thus that $g$ is explicit as well. Then the factorisation $g = g^\ast \prod_{i=1}^r p_i^{\lambda_i} $ into its irreducible factors is also explicit, and we get all its irreducible factors $p_i$ with $\gcd(\deg(p_i), nd)=1$ and $\gcd(\lambda_i,d)=1$; for each such factor, $L = k[t]/(p_i(t))$ is explicitly computed.
\end{proof}

As an immediate corollary, we recover the following well-known result.

\begin{cor}\label{32}  Let $k$ be any field. Let $\varphi $ be a cubic form on a finite-dimensional vector space $V$ over a field $k$.  Suppose that $\varphi_K$ is isotropic for some quadratic extension $K/k$. Then $\varphi$ is already isotropic over $k$.
\end{cor}
\begin{proof} Here $(d,n)=(3,2)$ and $nd-n-d= 1$. Since the only partition of 1 is $[1]$ and is a good partition, we can apply Theorem \ref{main} to conclude that $\varphi$ is already isotropic over $k$, as required.
\end{proof}

\subsection{Improving Coray's and Ma's results.} We now specialise  Theorem \ref{main} to the case when $(d,n) = (3,4)$. An important application of this is an improvement upon the results by Coray and Ma, as mentioned in the introduction.

\begin{thm}\label{34}
 Let $k$ be a field and let $\varphi$ be a cubic  form  on a finite-dimensional vector space $V $ over $k$.  If there exists a simple extension $K/k$ with  $[K:k] = 4$ such that $\varphi_K$ is isotropic, then there exists a finite extension $L/k$ with $[L:k] \in \{1,5\}$ such that $\varphi_L$ is isotropic. 
\end{thm}

\begin{proof} Here $(d,n)=(3,4)$ and $nd-n-d = 5$. Any partition of $5$ having a 1 in it is a good partition. The only partition of 5 not involving a 1 is $[3,2]$, which is also a good partition. Hence, by Theorem \ref{main}, there is some extension $L/k$ with $\varphi_L$ isotropic and $\gcd([L:k], 3)=1$, $4 \nmid [L:k]$, and $[L:k] \leq 5$. This leaves $[L:k] \in \{1, 2, 5\}$ as possibilities. But, by Corollary \ref{32}, if $[L:k] = 2$, then $\varphi$ is already isotropic over $k$. Hence, we can always find some $L/k$ with $\varphi_L$ isotropic and $[L:k] \in \{1,5\}$, as required.
\end{proof}

\subsection{The case when $d$ and $n$ are both prime} When $d$ and $n$ are both prime, we can refine the statement of Theorem \ref{main}. Indeed, when $d$ and $n$ are prime, the partitions of elements $u \in S_{d,n}$ are always good, as the following lemma shows. 

\begin{lemma} \label{admiss} Let $d\geq 3$ and $n \geq 2$ be prime with $\gcd(d, n)=1$. Let $u \in S_{d,n}$ and let  $[a_1,  ..., a_r]$ be a partition of $u$. Then there is some $i \in \{1, ..., r\}$ with $\gcd(a_i, nd)=1$. 
\end{lemma}
\begin{proof} We partition  $\{1, ..., r\}$ with the subsets
\[ \begin{array}{ll}
I_n&:=\{ i \in \{1, ..., r\}: n \mid a_i , \ d \nmid a_i\}, \\
 I_d&:=\{ i \in \{1, ..., r\}:  \ d \mid a_i\},\\
 I&:=\{ i \in \{1, ..., r\}: n \nmid a_i, \ d \nmid a_i\}.
\end{array} \]
If we can show that $I \neq \emptyset$, then we are done. We write
\[ u = \sum_{i =1}^r a_i = \sum_{i \in I_n} a_i + \sum_{i \in I_d} a_i +\sum_{i \in I} a_i.\]
Assume, for a contradiction, that $I = \emptyset$. We write $ \sum_{i \in I_n} a_i  = n c$, for some $c \in \Z_{\geq 0}$. Since $u \in S_{d,n}$ and $\gcd(d,n)=1$, we have that $u \equiv -n \bmod d$, which implies that  \( -n \equiv  \sum_{i \in I_n} a_i \equiv nc \bmod d.\) Hence, since $\gcd(d,n)=1$, we have that $c+1 \equiv 0 \bmod d$.
But since $nc \leq u < 	(d-1) n$, it follows that $c < d-1$. Hence, we have on the one hand that $0 \leq c \leq d-2$ and on the other hand that $c+1 \equiv 0 \bmod d$, which is a contradiction. Hence, $I \neq \emptyset$, as required.
 \end{proof}

As a corollary of Theorem \ref{main}, we get the following refinement.
\begin{cor} \label{mainprime} Let $k$ be any field. Let $\varphi $ be a degree $d$ form on a finite-dimensional vector space $V $ over a field $k$, where $d$ is an odd prime.  If $\varphi_K$ is isotropic for some extension $K/k$ with $n:=[K:k]$ prime and $\gcd(n,d)=1$, then $\varphi_L$ is also isotropic for some extension $L/k$ with $\gcd([L:k], nd) = 1$ and $[L:k] \leq nd -n-d$, where a set $\frakD$ of possible values for $[L:k]$ is explicitly computable. Moreover, if a non-trivial $K$-solution for $\varphi_K$  is known explicitly, then $L/k$ can be explicitly computed  as well.   
\end{cor}
\begin{proof} This follows immediately by Theorem \ref{main}, since $\frakP_{\textrm{bad}}(nd-n-d) = \emptyset$  by Lemma \ref{admiss}.
\end{proof}

%%%%%%%%%%%%%

\bibliographystyle{alpha}
\bibliography{refs}

\end{document}